\documentclass[12pt,a4paper]{article}

\usepackage{amsmath}
\usepackage{amsfonts}
\usepackage{amssymb}
\usepackage{amsthm}
\usepackage{enumerate}

\theoremstyle{definition} 
\newtheorem{thm}{Theorem}[section]

\newtheorem{prop}[thm]{Proposition}
\newtheorem{cor}[thm]{Corollary}

\theoremstyle{definition}
\newtheorem{defn}[thm]{Definition}

\newtheorem{exmp}[thm]{Example}

\theoremstyle{remark}
\newtheorem{rem}[thm]{Remark}

\begin{document}

\title{Sets of uniqueness for uniform limits of polynomials in several complex variables}
\author{K. Makridis  and  V. Nestoridis}
\date{19 April 2013}
\maketitle

\begin{abstract}
\noindent 
We investigate the sets of uniform limits $A(\overline{B}_n)$, $A(\overline{D}^I)$ of polynomials on the closed unit ball $\overline{B}_n$ of $\mathbb{C}^n$ and on the cartesian product $\overline{D}^I$ where $I$ is an arbitrary set and $\overline{D}$ is the closed unit disc in $\mathbb{C}$. We introduce the notion of set of uniqueness for $A(\overline{D}^I)$ (respectively for $A(\overline{B}_n)$) for compact subsets $K$ of $T^I$ (respectively of $\partial \overline{B}_n$) where $T=\partial D$ is the unit circle. Our main result is that if $K$ has positive measure then $K$ is a set of uniqueness. The converse does not hold. Finally, we do a similar study when the uniform convergence is not meant with respect to the usual Euclidean metric in $\mathbb{C}$, but with respect to the chordal metric $\chi$ on $\mathbb{C} \cup \{\infty \}$.  
\end{abstract}

\noindent 
\textbf{A.M.S classification No}
\\
primary 32A30, 30E10, 28A35
\\
secondary 30J99, 32A38, 46G20
~

\section{Introduction}

If $D$ is the open unit disc in $\mathbb{C}$ and $\overline{D}$ its closure, then the set of the uniform limits on $ \overline{D} $ of polynomials (with respect to the usual Euclidean metric in $\mathbb{C}$) is the well-known disc algebra $A(\overline{D})$; that is the set of all functions $f:\overline{D} \to \mathbb{C}$ continuous on $\overline{D}$ and holomorphic in $D$.

By Privalov's theorem, a compact set $K\subseteq \partial D = T$ with positive measure is a set of uniqueness for $A(\overline{D})$; that is if $f,g \in A(\overline{D})$ coincide on $K$, then they coincide on $\overline{D}$. This notion of set of uniqueness is compatible with the ones in \cite{Car.} and \cite{Hru.}. In fact, the converse also holds: a compact set $K\subseteq T$ is a set of uniqueness for $A(\overline{D})$ if and only if $K$ has a positive measure.

\noindent We extend some of the previous results in several complex variables, when $\overline{D}$ is replaced by $\overline{D}^I$ ($I$  arbitrary set) or the unit ball $\overline{B}_n$ of $\mathbb{C}$.

First, we investigate the set of the uniform limits of polynomials. Of course, every polynomial depends on a finite number of variables, even if $I$ is infinite. Thus, we find the classes $A(\overline{D}^I)$ and $A(\overline{B}_n)$ respectively. The class $A(\overline{D}^I)$ contains exactly all functions $f:\overline{D}^I \to \mathbb{C}$ continuous on $\overline{D}^I$ (where $\overline{D}^I$ is endowed with the cartesian topology) which separately as functions of each variable belong to $A(\overline{D})$.
The class $A(\overline{B}_n)$ contains exactly all\\ functions $f:\overline{B}_n \to \mathbb{C}$ continuous on $\overline{B}_n$ and holomorphic in the unit ball $B_n$.
By Hartogs's theorem, if $I$ is finite, this implies that $f$ has a power series development in $D^I$. In the case where $I$ is infinite we do not need power series expansions on $D^I$, since separate holomorphicity and continuity are sufficient and necessary for our purposes.

We consider $T^I$ $(T= \partial D)$ the distinguished boundary of $ \overline{D}^I$ and we introduce the notion when a compact set $K\subseteq T^I $ is a set of uniqueness for $A(\overline{D}^I)$. Our main result is that if a compact set  $K\subseteq T^I $ has positive measure (with respect to the natural measure on $T^I$), then $K$ is a set of uniqueness for $A(\overline{D}^I)$. This is based on Privalov's theorem \cite{Koo.} combined with some versions of Fubini's theorem \cite{Jes.}. We also give some examples of compact sets $K \subseteq T^I$ with zero measure which are also sets of uniqueness for  $A(\overline{D}^I)$, provided that $I$ contains at least two elements.

The boundary $\partial \overline{B}_n$ of the ball of $\mathbb{C}^n$  also carries a natural measure. We introduce the notion of set of uniqueness of $A(\overline{B}_n)$  (being compact subsets of $\partial \overline{B}_n$) and we prove that if $K \subseteq \partial \overline{B}_n$ has a positive measure, then $K$ is a set of uniqueness for $A(\overline{B}_n)$. If $n \geq 2$ the converse fails.

Next, we repeat all the previous study by replacing the usual Euclidean metric on $\mathbb{C}$ by the chordal metric $\chi$ on $\mathbb{C} \cup \{\infty\}$. We investigate the set of uniform limits on  $\overline{D}^I$ or $\overline{B}_n$ of polynomials with respect to $\chi$. Thus, we find the classes $\tilde A(\overline{D})$ , $\tilde A(\overline{D}^I)$ and $\tilde A(\overline{B}_n)$. The class $\tilde A(\overline{D})$ contains $A(\overline{D})$ and is strictly larger, because it contains the function $\frac{1}{1-z}$ which does not belong to $A(\overline{D})$. The precise statement is that a function $f: \overline D \to  \mathbb{C} \cup \{\infty\}$ belongs to
$\tilde A(\overline{D})$ if and only if $f \equiv \infty$ or if $f$ is continuous on $\overline D$, $f(D) \subseteq \mathbb{C}$ and $f$ is holomorphic in $D$ (\cite{Bro.Gau.Her.}, \cite{Nes.}).

A compact set $K \subseteq T= \partial \overline D$ is a set of uniqueness for $\tilde A(\overline{D})$ if and only if it has positive measure. Furthermore, the class $\tilde A(\overline{D}^I)$ contains exactly all functions $f: \overline{D}^I \to \mathbb{C} \cup \{\infty\}$ continuous on $\overline{D}^I $  (where $\overline{D}^I $ is endowed with the cartesian topology), which separately for each variable belongs to $\tilde A(\overline{D})$. 

We introduce the notion of a set of uniqueness for $\tilde A(\overline{D}^I)$ for compact subsets $K \subseteq T^I (T= \partial D)$ and we prove that if $K$ has positive measure, then it is a set of uniqueness for $\tilde A(\overline{D}^I)$. If $I$ contains at least two elements, the converse fails. Since $ A(\overline{D}^I) \subseteq \tilde A(\overline{D}^I)$, every set of uniqueness for $\tilde A(\overline{D}^I)$ is also a set of uniqueness for $A(\overline{D}^I)$. We do not know if the converse holds.

If we endow $A(\overline{D}^I)$ and $\tilde A(\overline{D}^I)$ with their natural metrics they become complete metric spaces. In fact, $A(\overline{D}^I)$ is a Banach algebra. Furthermore, the relative topology of $A(\overline{D}^I)$  from $\tilde A(\overline{D}^I)$ coincides with the natural topology of $A(\overline{D}^I)$ and $A(\overline{D}^I)$ is open and dense in  $\tilde A(\overline{D}^I)$.

Finally, we obtain similar results when $\overline{D}^I$ is replaced by $\overline{B}_n$. We notice that in the proof of the main results for $\overline{B}_n$ we use the analogue result for $\overline{D}^I$.

We mention that our methods of proof of the main theorems lead us naturally to versions of Fubini's theorem for infinitely many variables, countable or uncountable, some of which are open problems (\cite{Jes.}, \cite{Mah.}).

Finally we give a few examples of functions belonging to the previous studied classes. Let $\displaystyle{f((z_j)^ \infty_{j =1})= \sum_{j=1}^\infty \frac{z_j}{j^2}}$\;for all  $(z_j)^ \infty_{j=1} \in \overline{D}^ \mathbb{N}$;\\ 
then $f \in A(\overline{D}^ \mathbb{N})$.

Let $g(z_1,z_2)= \frac{1}{1-z_1z_2}$; then $g \in \tilde A(\overline{D}^2)$. The previous function $f$ belongs to $A(\overline{D}^\mathbb{N})$ and its image is bounded; therefore, if $|c|$ is big enough, the function $c+f(z_1,z_2)$ does not vanish at any point of $A(\overline{D}^\mathbb{N})$. Then the function $\frac{c+f(z_1,z_2)}{1-z_1}$ also belongs to $\tilde A(\overline{D}^\mathbb{N})$ and depends on all variables $z_1,z_2,...$. What is a less trivial example of a function belonging to $\tilde A(\overline{D}^\mathbb{N})$? The class $A(\overline{B}_n)$ is well-known. What are non-trivial examples of functions belonging to $\tilde A(\overline{B}_n)$? What about the functions $\omega(z_1,z_2)= \frac{1}{1-z_1}$ and $T(z_1,...z_n)= \frac{1}{1-z_1^2-z_2^2-...z_n^2}$ with $(z_1,z_2,...,z_n) \in \partial \overline{B}_n$?

An open issue is to study the structure of the element of $\tilde A(\overline{D}^I)$ and $\tilde A(\overline{B}_n)$. The cases of $A(\overline{D}^I)$ and $A(\overline{B}_n)$ have already been studied if $I$ is a finite set. What happens if $I$ is an infinite set? What is a characterization of the zero sets of elements of $\tilde A(\overline{D}^I)$, $\tilde A(\overline{B}_n)$ and $A(\overline{D}^I)$, $A(\overline{B}_n)$ when $I$ is infinite? What can be said about compact sets of interpolation for the previous classes? What about peak sets or null-sets? (see \cite {Rud2.}, \cite{Rud3.}).

\section{Preliminaries}

Mergelyan's theorem states that if  $K \subseteq \mathbb{C}$ is a compact set with $\mathbb{C} \setminus K$ connected, then every function $f \in A(K)$ can be uniformly approximated on $K$ by polynomials with respect to the usual Euclidean metric on $\mathbb{C}$ (\cite{Rud1.}); where $A(K)$ contains exactly all continuous functions $f:K \to \mathbb{C}$ which are holomorphic in $K^ \circ$ (if $K^ \circ= \emptyset$, then $A(K)= \mathbb{C}(K)$). It follows easily that if $K \subseteq \mathbb{C}$ is compact and $\mathbb{C} \setminus K$ is connected, then $A(K)$ coincides with the set of uniform limits of polynomials on $K$ with respect to the usual Euclidean metric on $\mathbb{C}$.

The proof of Mergelyan's theorem is complicated in the general case; however if $K$ is a compact disc with radius $r, 0<r<+ \infty$, the proof is elementary.

Let $D$ denote the open unit disc and $\overline D$ its closure. Let $T= \partial D$ denote the unit circle and $\frac{d \theta}{2 \pi}$ be the normalised one-dimensional Lebesgue measure on $T$. Then we have the following theorem.

\begin{thm}
Let $f$ be a holomorphic function on $D$ and $J \subseteq T$ be a measurable set with strictly positive one-dimensional measure (length). If, for every $\zeta \in J$ the non-tangential limit of $f(z)$, as $D\ni z \to \zeta$, exists and equals to zero, then $f \equiv 0$.
\end{thm}

\begin{defn}
Let $K \subseteq T$ be a compact set. Then $K$ is called a set of uniqueness for $A(\overline D)$, if for any $f,g \in A(\overline D)$, the following holds.
If $f_{|K}= g_{|K}$ then $f \equiv g$.
\end{defn}

It follows from Theorem $2.1$ that if $K \subseteq T$ is a compact set with positive measure, then $K$ is a set of uniqueness for $A(\overline D)$. If $F \subseteq T$ is a compact set with zero measure, then there exists a peak-function $\phi \in A(\overline D)$, such that $\phi _{\rvert F} \equiv 1$ and $|\phi (z)|<1$ for all $z \in \overline{D} \setminus {F}$ (\cite{Hof.}). It follows easily that $F$ is not a set of uniqueness for $A(\overline D)$. Therefore, the following holds.

\begin{thm}
Let $K \subseteq T$ be a compact set. Then $K$ is a set of uniqueness for $A(\overline D)$ if and only if $K$ has positive measure.
\end{thm}

In complex analysis of one variable very often we are dealing with functions taking the value $\infty$,  as well. In order to talk about uniform convergence (approximation) for such functions we need to replace the usual Euclidean metric on $\mathbb{C}$ by a metric on $\mathbb{C} \cup \{\infty\}$. Since all metrics on the compact space $\mathbb{C} \cup \{\infty\}$ compatible with the usual topology of $\mathbb{C} \cup \{\infty\}$ are uniformly equivalent, it suffices to work with any such metric.
Such a very well-known metric is the chordal metric $\chi$ on $\mathbb{C} \cup \{\infty\}$. This is the metric induced on $\mathbb{C} \cup \{\infty\}$ via stereographic projection by the Euclidean metric of $\mathbb{R}^3$ restricted to $S^2 = \{ (x_1,x_2,x_3) \in\mathbb{R}^3 : x_1^2+x_2^2+x_3^2 = 1 \}$. Then,

\begin{itemize}
\item $\chi (a,b) = \frac{|a-b|}{\sqrt{1+|a|^2} \sqrt{1+|b|^2}}$ \; if $a,b \in \mathbb{C}$

\item $\chi (a,\infty ) =\chi (\infty, a) = \frac{a}{\sqrt{1+|a|^2}}$ \; if $a \in \mathbb{C}$

\item and $\chi (\infty, \infty) = 0$.
\end{itemize}

It is an open question what is the general form that Mergelyan's theorem takes, if the usual Euclidean metric on $\mathbb{C}$ is replaced by the chordal metric in $\mathbb{C} \cup \{\infty\}$ (\cite{Nes.}). However, in particular cases we know the answer. Such a case is when $K$ is a compact disc with radius $r, 0<r<+ \infty$.

\begin{defn}
A function $f: \overline{D} \to \mathbb{C} \cup \{\infty\}$ belongs to the class $\tilde A( \overline{D})$ if and only if $f$ is identically equal to $\infty$ or $f(D) \subseteq \mathbb{C}, \;  f_{|D}$ is holomorphic and for every $\zeta \in T= \partial D$ the limit 
$\displaystyle{\lim_{z \to \zeta}{f(z)}}$ exists in $\mathbb{C} \cup \{\infty\}$.
\end{defn}

\begin{rem} It is equivalent to say that $f \in \tilde A( \overline{D})$ if and only if $f \equiv \infty$ or $f: \overline{D} \to \mathbb{C} \cup \{\infty\}$ is continuous, $f(D) \subseteq \mathbb{C}$ and $f_{|D}$ is holomorphic. Further $A(\overline{D}) \subseteq \tilde A( \overline{D})$ and $A(\overline{D}) \neq \tilde A( \overline{D})$ because the function $f(z)= \frac{1}{1-z}$ belongs to $\tilde A( \overline{D}) \setminus A(\overline{D})$.
\end{rem}

\begin{thm}(\cite{Nes.})
Let $f: \overline{D} \to \mathbb{C} \, \cup \, \{\infty\}$ be any function. Then $f \in \tilde A( \overline{D})$ if and only if, there exists a sequence $p_n$ of polynomials of one variable such that 
\[
\sup \{\chi (p_n(z),f(z)) : |z| \leq 1 \}
 \xrightarrow{\text{$n \to  + \infty$}} 0.
\]

\end{thm}

Thus, the function $\frac{1}{1-z}$ can be approximated uniformly on $D$ by polynomials with respect to the chordal metric $\chi$ on $\mathbb{C} \cup \{\infty\}$, but not with respect to the usual Euclidean metric on $\mathbb{C}$.
\\
Privalov's theorem implies that if $f,g \in \tilde A(\overline D)$ coincide on a set $K \subseteq T$ of positive measure, then $f \equiv g$ (\cite{Nes.}). If $F \subseteq T$ is a compact set with zero measure, then there is a peak function $\phi \in A(\overline D) \subseteq \tilde A(\overline D)$, $\phi_{|F} \equiv 1$, $|\phi (z)|<1$ for all $z \in \overline D \setminus F$ (\cite{Hof.}). Thus, the functions $\phi$ and $1$ belong to $\tilde A(\overline D)$, they coincide on $F$ but they are not equal.

\begin{defn}
Let $K \subseteq T$ be a compact set. The set K is called a set of uniqueness for $\tilde A(\overline D)$, if the following holds.
If $f,g \in \tilde A(\overline D)$ coincide on $K$, then $f \equiv g$.
\end{defn}
By the previous discussion we have proved the following.

\begin{thm}
A compact set $K \subseteq T$ is a set of uniqueness for $\tilde A(\overline D)$, if and only if $K$ has positive measure; thus, compact sets of uniqueness for $\tilde A(\overline D)$ coincide with those  for $A(\overline D)$.
\end{thm}

We notice that in a similar way we can define $\tilde A(\overline D(w,r))$ for any compact disc with center $w \in \mathbb{C}$ and radius $r$, $0<r<+ \infty$.

Our aim is to extend some of the previous results in several complex variables and precisely to the case of polydiscs or Euclidean balls. Let $I$ be any set with at least two points. We consider the cartesian product $\overline D^I$ endowed with the cartesian topology. Then basic open sets are of the form $\prod_{i \in I} Y_i$, where all $Y_i \subseteq \overline D$ are open in the relative topology of $\overline{D}$ and $Y_i = D$ for all $i \in I$ except a finite number of $i$'s. It follows easily that if $\zeta \equiv ({\zeta}_i)_{i \in I} \in \overline D^I$ is fixed, then the set 
 $\bigcup_{F \subseteq I, \text{F fitite}} \overline D^F \times \prod_{i \in I \setminus F} \{\zeta_i\}$ 
is dense in $\overline D^F$. If $I$ is a finite or infinite denumerable set, then $\overline D^I$ is metrizable and according to Tychonoff's theorem, compact; thus, every continuous function $f: \overline D^I \to Y$ is uniformly continuous for any metric space $Y$; in particular for $Y = \mathbb{C}$ or $Y = \mathbb{C} \cup \{\infty\}$. If $I$ is a non-denumerable set, we do not have a metric at our disposal.

However, $\overline D^I$ carries a uniform structure and, as it is again compact by Tychonoff's theorem, one could try to prove that every continuous function $f: \overline D^I \to Y$ ($Y = \mathbb{C}$ or $\mathbb{C} \cup \{\infty\}$) is automatically uniformly continuous. We will avoid to follow this way, but we will prove directy the following.

\begin{prop}
Let I be an infinite set denumerable or uncountable. Let $f: \overline D^I \to Y$ be a continuous function where $(Y,d)$ is a metric space and let $\varepsilon >0$ be given. Then, there exists a finite set $F \subseteq I$, such that for every point $\zeta = (\zeta_i)_{i \in I} \in \overline D^I$ the function $g: \overline D^I \to Y$ defined by $g(z) = f(w(z))$, where $w(z) = ((w(z))i)_{i \in I}$ with 
\begin{equation*}
(w(z))_i = 
\begin{cases}
\zeta_i & \text{for } i \in I \setminus F \\
z_i & \text{for } i \in F
 \end{cases}
\end{equation*}

$z = (z_i)_{i \in I}$ satisfies $d(f(z),g(z)) < \frac{\varepsilon}{2}$ for all $z \in \overline D^I$.
\end{prop}

\begin{proof}
Since $f$ is continuous on $\overline D^I$, for every $\tau \in \overline D^I$ there exists a basic open neighbourhood $V_{\tau}$ such that $d(f(\tau),f(\sigma)) < \frac{\varepsilon}{4}$ for all $\sigma \in V_{\tau}$. By compactness we have $V_{\tau ^1} \cup ... \cup V_{\tau ^m} = \overline{D}^I$ for some finite $m$ and ${\tau ^1}, \dots ,{\tau ^m} \in \overline{D}^I$. For every $k \in \{1, ... , m\}$ there exists a finite set $F_k \subseteq I$ and $\delta ^k >0$ so that $V_{\tau ^k} = \{z = (z_i)_{i \in I}\;|\;|z_i - \tau_i^k| < \delta ^k$ for $i \in F_k \}$.

We set $F = F_1 \cup ... \cup F_m$ and let $\zeta = ({\zeta}_i)_{i \in I}$ be arbitrary in $\overline{D}^I$. Let $z = (z_i)_{i \in I}$ and set $[w(z)]_i = {\zeta}_i$ for $i \in I \setminus F$ and $[w(z)]_i = z_i$ for $i \in F$. Because $V_{\tau ^1}\cup \dots \cup V_{\tau ^m} = \overline{D}^I$, there exists a $k \in \{1, ... , m\}$ such that $z \in V_{\tau ^k}$. Then, for $i \in F_k$ we have $|z_i - \tau_i^k| < \delta^k$. It follows that $|[w(z)]_i - \tau_i^k| < \delta^k$ for all $i \in F_k$ because $F_k \subseteq F$ and $(w(z))_i = z_i$ for all $i \in F$. Thus, $w(z) \in V_{\tau ^k}$. Therefore, $d(f(w(z)),f(\tau^k)) < \frac{\varepsilon}{4}$.
Since $z \in V_{\tau ^k}$, we also have $d(f(z),f(\tau^k)) < \frac{\varepsilon}{4}$. The triangle inequality implies that $d(f(w(z)),f(z)) < \frac{\varepsilon}{2}$ since $g(z) = f(w(z))$. The proof is complete.
\end{proof}

An immediate corollary of proposition 2.9 is the following well-known fact.

\begin{cor}
let $I$ be an uncountable set $f: \overline{D}^I \to Y$ be a continuous function, where $(Y,d)$ is a metric space. Then $f$ depends on a denumerable set of coordinates.
\end{cor}

\begin{proof}
For every $n = 1,2, ...$ there exists a finite set $F_n \subseteq I$ and a function $g_n: \overline{D}^I \to Y$ depending only on the coordinates in $F_n$ such that $d(f(z),g_n(z)) < \frac{1}{n}$ for all $z \in \overline{D}^I$. This holds because of Proposition 2.9. We set $F = \cup_{n = 1}^{\infty} F_n$ and we get the result.
\end{proof}

Finally we will also use Hartogs's theorem \cite{Nar.}. For an open set $\Omega \subseteq \mathbb{C}^n$, if we consider any function $f: \Omega \to \mathbb{C}$, then $f$ is locally represented in $\Omega$ by a power series with any center $\zeta \in \Omega$ absolutely and uniformly convergent on any closed polydisc contained in $\Omega$ with center $\zeta$, if and only if $f$ is separately holomorphic with respect to each variable. Then, the convergence is absolute and uniform on each compact Euclidean ball contained in $\Omega$ as well (\cite{Rud2.},\cite{Rud3.}).

Finally we mention that, if $\{X_i\}_{i \in I}$ are compact spaces endowed with regular 
Borel probability measures $\{\mu_i\}_{i \in I}$, then we can define a regular Borel probability measure $\mu = \prod_{i \in I} \mu_i$ on the cartesian product $\prod_{i \in I} X_i$\\
endowed with the cartesian topology (\cite{Fol.}). If $I$ is finite we have the usual Fubini's theorem. If $I$ is infinite denumerable we have the following result of Jessen (\cite{Jes.}).

\begin{thm}
Under the above assumptions, let $I = \mathbb{N}$ and let $f \in L^1(\mu)$. Then there exists a measurable set $J \subseteq \prod_{i \in \mathbb{N}} X_i$ with $\mu (J) = 1$ and if for every $\zeta \in J$, $n \in \mathbb{N}$ and every $z \in \prod_{i \in \mathbb{N}} X_i$ we set $g_{\zeta ,n}(z) = f(w_{\zeta ,n}(z))$, where $[w_{\zeta, n}(z)]_m = z_m$ for $m \leq n$ and $[w_{\zeta, n}(z)]_m = \zeta_m$ for $m > n$, then we have 

$$\int g_{\zeta ,n}(z) \mathrm{d} \mu(z) \to \int f(z) \mathrm{d}\mu(z)$$ for all $\zeta \in J$.

\end{thm}

We will apply these results in the case where $X_i = T$ is the unit circle and $\mu_i$ is the normalised Lebesgue measure $\frac{\partial \theta}{2\pi}$ on $T$.

Extensions of Theorem 2.11 have been obtained by D.Maharam (\cite{Mah.}) in the case where $I$ is uncountable (and well ordered). However, we do not know if such results hold when the limit is taken with respect to the directed set of all finite subsets of $I$ with the order of the relation of inclusion. We can prove the following:

''Let $\{X_i\}_{i \in I}$,$\{\mu_i\}_{i \in I}$ be as above, where $I$ is an infinite set. Let $f: \prod_{i \in I} X_i \to \mathbb{R}$ be a continuous function. For every $\zeta \in \prod_{i \in I}X_i$ $(\zeta = (\zeta_i)_{i \in I})$ and every finite set $F \subseteq I$ we set $g_{\zeta, F}(z) = f(w_{\zeta, F}(z))$ where $[w_{\zeta, F}(z)]_i = z_i$ if $i \in F$ and $[w_{\zeta, F}(z)]_i = \zeta_i$ if $i \in I \setminus F$. Let $\varepsilon >0$. Then there exists a finite set $F_{\varepsilon} \subseteq I$, such that for every finite set $F$, with $F_{\varepsilon} \subseteq F \subseteq I$ and every $\zeta \in \prod_{i \in I} X_i$ we have $|\int g_{\zeta, F}(z) \mathrm{d} \mu(z) - \int f(z) \mathrm{d} \mu(z)| < \varepsilon$.''

We do not know if the previous result remains valid if $f$ is the characteristic function $f = \chi_K$ of any compact subset $K \subseteq \prod_{i \in I} X_i$, or more generally when $f \in L^1(\mu)$. Certainly for such an extension one would require that the result holds for almost all $\zeta \in \prod_{i \in I} X_i$ and that the set $F_{\varepsilon}$ depends on $\zeta$ as well. This is still open, even if $I = \mathbb{N}$.

If we require the weaker result that for almost for all $\zeta \in \prod_{i \in I} X_i$ and for every $\varepsilon > 0$ and every finite set $F \subseteq I$ there exists a finite set $F'$ with $F \subseteq F \subseteq 'I$ such that  $|\int g_{\zeta, F'}(z) \mathrm{d} \mu(z) - \int f(z) \mathrm{d} \mu(z)| < \varepsilon$, then the situation is the following. For $I$ countable this is true, as it follows from Theorem 2.11. In the uncountable case we do not know the answer.

\section{Sets of uniqueness for the algebra of a polydisc}

In this section we consider the compact space $\overline{D}^I$, where $\overline{D} = \{z \in \mathbb{C} : |z| \leq 1 \}$ and $I$ any non-empty set endowed with the cartesian topology. A polynomial on $\overline{D}^I$ is a finite sum of monomials 
$ P(z) = \sum_{a \in F} c_a z_{i_{1} (a)}^{k_{1} (a)} \cdots z_{i_{M(a)} (a)}^{k_{M(a)} (a)} $, where $F$ is a finite set, $M(a) \in \mathbb{N}$ for any $a \in F$, $i_j (a) \in I$ and $k_j (a) \in \mathbb{N}$ for all $j = 1, \dots , M(a), a \in F, c_a \in \mathbb{C}$ for all $a \in F$ and $|z_{i_j (a)}| \leq 1$ for all $j = 1, \dots , M(a), \; a \in F$ and $z = (z_i)_{i \in I}$, $|z_i| \leq 1$ for all $i \in I$.

Therefore, $P$ depends on a finite number of variables even if the set $I$ is infinite. We investigate all possible uniform limits of sequences of polynomials $P_n$ on $\overline{D}^I$ with respect to the usual  metric on $\mathbb{C}$, where the degree of each $P_n$ and the set of variables on which each $P_n$ depends may vary. 

\begin{prop}
Let $f:\overline{D}^I \to \mathbb{C}$ be a function. Then there exists a sequence $\{p_n\}_{n \geq 1}$ of polynomials converging uniformly towards $f$ on $\overline{D}^I$ with respect to the usual Euclidean metric of $\mathbb{C}$, if and only if $a)$ and $b)$ below are satisfied, where 
\begin{enumerate}[a)]
\item $f$ is continuous on $\overline{D}^I$ and
\item for every $i_0 \in I$ and every $z=(z_i)_{i \in I} \in  \overline{D}^I$ the function $\overline{D} \in w \to f(\zeta(w)) \in \mathbb{C}$ belongs to $A(\overline{D})$, where $\zeta_{i_0}(w)=w$ and $\zeta_{i}(w)=z_i$ for all $i \in I \setminus \{ i_0 \}$.
\end{enumerate}

\end{prop}

\begin{proof}
It is immediate that if $p_n \to f$, then $f$ satisfies $a)$ and $b)$.

Suppose that $f$ satisfies $a)$ and $b)$ and let $\varepsilon > 0 $ be given. It suffices to find a polynomial $p$ on $ \overline{D}^I$, so that $|f(z) - p(z)| < \varepsilon$ for all $z \in  \overline{D}^I$.

We consider first the case where $I$ is finite; without loss of generality let $I = \{ 1, \dots , N \}$ and $f: \overline{D}^N \to \mathbb{C}$ satisfies $a)$ and $b)$. Since $\overline{D}^N$ is a compact metric space, it follows that $f$ is uniformly continuous. Therefore, there exists $ r \in (0,1)$, such that $|f(z) - f(rz)| < \frac{\varepsilon}{2}$ for all $z \in \overline{D}^N$. By $b)$ combined with Hartogs's theorem \cite{Nar.}, the function $f$ has a local power series expansion with center $0 = (0, \dots , 0)$. The convergence towards $f$ is absolute and uniform on any closed polydisc with center $0$ which is contained in the open domain of definition of the holomorphic function $f$. In particular, $a$ partial sum $Q$ of this extension satisfies  $|Q(w) - f(w)| < \frac{\varepsilon}{2}$ for all $w=(w_1, \dots, w_N)$ with $|w_j| \leq r$, because $r<1$. It follows easily that $|p(z) - f(z)|< \varepsilon $ for all $z \in \overline{D}^N$, where $p(z)$ is the polynomial $p(z) = Q(rz)$. This finishes the proof in the particular case where $I$ is finite.

Now assume that $I$ is infinite, that $f$ is defined on $\overline{D}^N$ and satisfies $a)$ and $b)$. Let $\varepsilon > 0$. It suffices to find a polynomial $p$ satisfying $|p(z) - f(z)| < \varepsilon$ for all $z \in \overline{D}^N$. By Proposition 2.9 we can find a point $\zeta \in \overline{D}^N$ and a finite set $F \subset I$ such that, the function $g: \overline{D}^I \to \mathbb{C}$, defined by $g(z) = f(w(z))$, where $w(z)_i = {\zeta}_i$ if $i \in I \setminus F $ and $w(z)_i = z_i$ for $i \in F$, satisfies $|g(z) - f(z)| < \frac{\varepsilon}{2}$ for all $z \in  \overline{D}^I$. The function $g$ depends on a finite number of variables and defines a function on $\overline{D}^F$ satisfying $a)$ and $b)$. By the previous case we find a polynomial $p$ (on $\overline{D}^F$, which defines also polynomial on $\overline{D}^I$) such that $|p(z) - g(z)| < \varepsilon$ on $\overline{D}^I$ and the proof is complete.
\end{proof}
 
Myrto Manolaki asked if condition $b)$ of Proposition 3.1 implies condition $a)$ of Proposition 3.1. Paul Gauthier gave a negative answer in the harmonic case. The following counterexample on $\overline{D}^2$ for the holomorphic case was suggested by Greg Knese; considered the rational inner function $f$ on the bidisc, where $f(z,w) = \frac{2zw - z - w}{2-z-w}$ for $|z| \leq 1$, $|w| \leq 1$, $(z,w) \neq (1,1)$ and $f(1,1) = -1$. This function satisfies $b)$ of Proposition 3.1 but not $a)$; because for $z=\overline{w}= e^{i \theta}, \theta \in \mathbb{R}$ $f(e^{i \theta}, e^{-i \theta}) = 1 \neq -1 = f(1,1)$ and the function f is not continuous at $(1,1)$ seeing as as function defined on $\overline{D}^2$.

\begin{defn}
$A (\overline{D}^I)$ contains exactly all functions $ f: \overline{D}^I \to \mathbb{C}$ which satisfy $a)$ and $b)$ of Proposition 3.1.
\end{defn}

It follows easily that $A(\overline{D}^I)$ is a Banach algebra and contains exactly all uniform limits of polynomials on $\overline{D}^I$ with respect to the usual Euclidean metric on $\mathbb{C}$. By the maximum principle, it can easily be seen that $T^I$ where $T = \partial \overline{D} = \{ z \in \mathbb{C} : |z| = 1 \}$ is a boundary for $A(\overline{D}^I)$; that is, for every $f \in A(\overline{D}^I)$ there exists a point $\zeta \in T^I$ such that $|f(\zeta)| \geq |f(z)|$ for all $z \in \overline{D}^I$. Equivalently the restriction operator $\pi _{T^I} : A(\overline{D}^I) \ni f \to f_{|T^I} \in C( T^I)$ is an isometry.

\begin{defn} A compact set $K \subset T^I$ is called a set of uniqueness for the class $A(\overline{D}^I)$, if the operator $\pi _{K} : A(\overline{D}^I) \ni f \to f_{|K} \in C (K)$ is injective. Equivalently if $f,g \in A(\overline{D}^I)$ satisfy $f_{|K} = g_{|K}$, then $f \equiv g$ on $\overline{D}^I$. 
\end{defn}

\begin{rem} 
A compact set $K \subset T^I$ is called a peak-set (for $A(\overline{D}^I)$), if there exists a function $\phi \in A(\overline{D}^I)$ satisfying $\phi _{|K} \equiv 1$ and $| \phi (z)| <1$ for all $z \in \overline{D}^I \ K$ (see \cite{Rud2.}).


One can easily see that, if a compact set $K \subset T^I$ is a peak-set, then it is not a set of uniqueness for $A(\overline{D}^I)$. For $I = \{ 1,2 \}$ the set $K = \{ 1 \} \times T \subset T^2 \subset  \overline{D}^2$ is neither a peak-set nor a set of uniqueness for $A(\overline{D}^I)$. Considering the function $ f(z_1, z_2) = \frac{1+z_1}{2}$ we see that $f_{|K} \equiv 1$ but $f \not\equiv 1$; this yields that $K$ is not a set of uniqueness for $A(\overline{D}^2)$. Furthermore, if $g \in A(\overline D^2)$ satisfies $g_{|K} \equiv 1$ it follows easily that $g(1,0) = 1$; thus, K is neither a peak-set. The previous remarks can easily  be extended to the case of any arbitrary set $I$ containing at least two points. 
\end{rem}

\begin{thm}
Let $K \subset T^I$ be a compact set with $\lambda (K) >0 $, where $\lambda$ denotes the product measure on $T^I$ of the normalised Lebesgue measure $\frac{d \theta}{2 \pi}$ on each factor $T$ of $T^I$. Then $K$ is a set of uniqueness for $A(\overline{D}^I)$.
\end{thm}

\begin{proof}
If $I$ is a singleton, then the result follows from the Theorem of Privalov (Theorem 2.1, Theorem 2.3). If $I$ is finite, then the result follows by induction combined with Fubini's Theorem and the theorem of Privalov. Indeed, if $K \subset T^n$ has positive measure (of dimension n), then, by Fubini's Theorem there exists $J \subset T^{n-1}$ with positive measure (of dimension $n-1$) such that, for every $y \in J$ the set $\{ x \in T : (x,y) \in K \} = A^y$ has positive measure (of dimension 1).

If $f,g \in A(\overline{D}^n)$ satisfies $f_{|K} = g_{|K}$, then $f_{|A^y \times \{y\}} \equiv g_{|A^y \times \{y\}}$ for all $y \in J$. Since $A^y$ has positive measure , Privalov's theorem yields $f_{|\overline{D} \times \{y\}} \equiv g_{|\overline{D} \times \{y\}}$ and this for all $y \in J$. We find a compact set $R \subset J$ with positive measure. By the induction hypothesis $R$ is a set of uniqueness ; therefore $f_{| \{ z \} \times \overline{D}^{n-1} } \equiv g_{| \{ z \} \times \overline{D}^{n-1} }$ and this for all $z \in \overline{D}$. Therefore $f \equiv g$. This completes the proof in the case where $I$ is finite. 

Next assume that $I$ is infinite denumerable; without loss of generality $I= \mathbb{N}$. Let $K \subset T^I$ be a compact set with positive measure. Since the product measure is a regular Borel measure (\cite{Fol.}), the extension of Fubini's  theorem proved by Jessen (Theorem 2.11) applies to the characteristic function $\chi_K$, which is measurable and integrable. Therefore, there exists a measurable set $J \subset T^I$ of full measure, such that for every $\zeta \in J$, $\zeta=(\zeta_1, \zeta_2, \dots)$ the sequence $\lambda _n (E_n)$, $n=1,2, \dots$ with 
$E_n = \{ (z_1, \dots , z_n) \in T^n : (z_1, \dots , z_n, \zeta_{n+1}, \zeta_{n+2}, \dots) \in K \}$ converges to $\lambda (K)> 0$ (where $\lambda$ is the product measure on $T^N$ of $\frac{d \theta}{2 \pi}$ on each factor $T$ and $ \lambda _n$ is the product measure on $T^n$ ).Fix a $\zeta \in J$. Since there exists $n_o$, such that for all $ n \geq n_0$ the sets $E_n$ have positive measure our result in the finite case implies that each $E_n \subset T^n$ is a set of uniqueness for $A(\overline{D}^n)$, $n \geq n_0$. Let $f,g \in A(\overline{D}^{\mathbb{N}})$ satisfy $f_{|K} = g_{|K}$ then, the functions $f_n ,g_n \in A(\overline{D}^n)$ defined by $f_n(z_1, \dots , z_n) = f_n(z_1, \dots , z_n, \zeta_{n+1}, \zeta_{n+2}, \dots)$ and $g_n(z_1, \dots , z_n) = g_n(z_1, \dots , z_n, \zeta_{n+1}, \zeta_{n+2}, \dots)$ coincide on $E_n$. Since $E_n$ is a set of uniqueness on $A(\overline{D}^n)$, it follows $f_n = g_n$ for each $n$ and $f=g$ on the set $\displaystyle \bigcup_n A(\overline{D}^n) \times \{ \zeta_{n+1} \} \times \{ \zeta_{n+2} \} \times \cdots $ which is dense in $A(\overline{D}^n)$ for the cartesian topology. As $f$ and $g$ are continuous, it follows that $f \equiv g$. This completes the proof in the infinite denumerable case. 

Assume now that $I$ is infinite non-denumerable. It is well known that every continuous function $f$ on $A(\overline{D}^I)$ depends on a denumerable set of variables. Let $K \subset T^I$ be compact with positive measure $\lambda (K) >0$. Let $f,g \in A(\overline{D}^I)$ be such that $f_{|K} = g_{|K}$. We have to show that $f \equiv g$.\\
If $F \subset I$ is a denumerable set, such that $f$ and $g$ depend only on the coordinates in $F$, then $f$ and $g$ may be seeing as functions in $A(\overline{D}^F)$ and $f_{|S} = g_{|S}$ where $S \subset T^F$ is the projection of $K$ on $A(\overline{D}^F)$ which is a compact set. By the definition of the product measure we have $\lambda _F (S) \geq \lambda _I (K) >0$, because $K \subset S \times T^{I-F}$ and $T^{I-F}$ has measure $1$, where $\lambda _F$ and $\lambda _I$ are the product measures on $T^F$ and $T^I$, respectively.  Therefore, $\lambda _F (S) >0$.
By our result in the denumerable case it follows $f=g$ as element of $A(\overline{D}^F)$, which easily implies $f=g$ on $A(\overline{D}^I)$, since $f$ and $g$ depend only on the coordinates in $F$. This completes the proof.
\end{proof}

\begin{rem} In the case where $I$ is non-denumerable it can easily be seen that a compact set $K \subset T^I$ is a set of uniqueness for $A(\overline{D}^I)$ if only if, for every infinite denumerable set $F \subset I$ the projection of $K$ on $T^F$ is a set of uniqueness for $A(\overline{D}^F)$.
\end{rem}

\begin{exmp} If the set I contains at least two points, then there exists a compact set $K \subset T^I$ which is a set of uniqueness for $A(\overline{D}^I)$ and has zero measure.
\end{exmp}

We sketch such an example in $A(\overline{D}^2)$, but it can easily be transferred to the generic case, when $I$ contains more than one points. 

Let $J = [0,1]$, $J \cap \mathbb{Q}=\{ q_1, q_2, \dots \} $ and $a_n = \frac{1}{n}$, $n=1,2, \dots $. We set $A_n = \{ (e^{iq_n}, e^{iy}) : 0 \leq y \leq a_n \} \subset T^2 $ and $K = (\bigcup _{n=1}^{\infty} A_n) \cup \{ (e^{ix}, e^{i \cdot 0}) : x \in J \} \subset T^2$.
One can easily check that $K$ is compact with zero measure (of dimension 2). Let $f, g \in A \big( \overline{D}^2 \big)$ be such that $f_{|K} = g_{|K}$. Since the one-dimensional measure of $\{ e^{iy} : (e^{iq_n}, e^{iy}) \in A_n \} $ is $a_n = \frac{1}{n} >0 $ it follows by Privalov's theorem (Theorem 2.1, Theorem 2.3) that $f(e^{iq_n} , w) = f(e^{iq_n} , w)$ for all $w \in A(\overline{D})$. By the continuity of $f$ and $g$ and because $J \cap Q$ is dense on $J = [0,1]$, it follows that $f_{|S} = g_{|S}$ where $S = \{ (e^{ix}, e^{iy}) : 0 \leq x \leq y \; and \; 0 \leq y \leq 2\pi \}$. Since the 2-dimensional measure of $S$ is strictly positive, it follows by Theorem 3.5. that $S$ is a set of uniqueness for $A(\overline{D}^2)$. Therefore $f \equiv g$. The proof of Example 3.7 is complete.

\begin{prop} Let $I \neq \emptyset$ and $\big( I_j \big)_{j \in J}$ a partition of $I$. For every $j \in J$ let $K_j \subset T^{I_j}$ be a compact set of uniqueness for $A(\overline{D}^{I_j})$. Then the compact set $K = \prod_{j \in J} K_j \subset \prod_{j \in J} T^{I_j} \equiv T^I$ is a set of uniqueness for $A(\overline{D}^I)$.
\end{prop}

\begin{proof}
Consider first the case where $J$ contain two points: $J = \{ 1,2 \} $. Let $f,g \in A(\overline{D}^I)$ : $I= I_1 \cup I_2$, $I_1 \cap I_2 = \emptyset$. We assume that $f_{|K_1 \times K_2} = g_{|K_1 \times K_2}$, where $K_1 \subset T^{I_1}$ and $K_2 \subset T^{I_2}$ are sets of uniqueness of $A(\overline{D}^{I_1})$ and $A(\overline{D}^{I_2})$ respectively. Let $\zeta\in K$, $\zeta = (\zeta_1, \zeta_2)$, $\zeta_1 \in K_1$ and $\zeta_2 \in K_2$ be fixed. Let the functions $f(\cdot , {\zeta}_2) \in A(\overline{D}^{I_1})$, $g(\cdot , \zeta_2) \in A(\overline{D}^{I_1})$ coincide on $K_1$. Since $K_1$ is a set of uniqueness for $ A(\overline{D}^{I_1})$ it follows $f(z,\zeta_2) = g(z, \zeta_2)$ for all $z \in  A(\overline{D}^I)$ and this for all ${\zeta}_2 \in K_2$. Since $K_2$ is a set of uniqueness for $ A (\overline{D}^{I_2})$ it follows easily $f(z,w) = g(z,w)$ for all $ (z,w) \in   \overline{D}^I  =   \overline{D}^{I_1}  \times   \overline{D}^{I_2} $. Thus, $K = K_1 \times K_2$ is a set of uniqueness for $ A(\overline{D}^{I_1 \cup I_2})$.

If $J$ is a finite set, the result follows easily by induction because $K_1 \times \cdots \times K_n = (K_1 \times \cdots \times K_{n-1}) \times K_n$. \\
Suppose $J$ is infinite. Let $f,g  \in A(\overline{D}^I)$ coincide on $K = \prod_{j \in J} K_j$, we will show $f \equiv g$. Fix an element $\zeta = \big( \zeta_j \big) _{j \in J} \in \prod_{j \in J} K_j$, $\zeta_j \in K_j$. By the result in the finite case we conclude that $f_{|B} = g_{|B}$ where $$B = 
\bigcup_{\substack{
   F \subset J \\
   \text{finite set}
  }}
\overline{D}^{ \bigcup_{j \in F} I_j} \times \prod_{j \in J \setminus F} \{ \zeta_j \}
.$$ Since every finite subset of $I = \cup_{j \in J} I_j$ is contained in $\bigcup_{j \in F} I_j$ for some finite subset $F \subset J$, it follows that $B$ is dense in $\overline{D}^I$. By the continuity of $f$ and $g$, we conclude $ f \equiv g$. The proof is complete.
\end{proof}

\section{Spherical approximation on polydiscs}

In this section we consider, as in the previous section, product spaces $\overline{D}^I$ and we investigate the uniform limits of polynomials on $\overline{D}^I$ with respect to the chordal metric $\chi$ on $\mathbb{C} \cup \{\infty\}$.

\begin{prop}
Let $f: \overline{D}^I \to \mathbb{C} \cup \{\infty\}$ be a function. Then, there is a sequence of polynomials $\{p_n\}_{n \geq 1}$ such that 

\[
\sup \{\chi(p_n(z),f(z)) : z \in \overline{D}^I \}
 \xrightarrow{\text{$n \to  + \infty$}} 0
\]
if and only if a') and b') below are satisfied, where:
\begin{enumerate}[a')]
\item the function $f$ is continuous

\item for every $j_0 \in I$ and every $z = (z_i)_{i \in I} \in \overline{D}^I$, the function $ \overline{D} \ni w \to f(\zeta(w)) \in \mathbb{C} \cup \{\infty\}$   belongs to $\tilde A(\overline{D})$, where $[\zeta(w)]_{j_o} = w$ and $[\zeta(w)]_{j} = z_j$ for all $j \in I \setminus \{j_o\}$.
\end{enumerate}

\end{prop}

The proof is similar to the proof of Proposition 3.1. The only difference is in the finite case, where $I$ contains exactly $N$ points ($N \in \mathbb{N}$), we have $\chi (f(z) , f(rz)) < \frac{\varepsilon}{2}$ instead of $|f(z) - f(rz)| < \frac{\varepsilon}{2}$. Further, the inequality $|Q(w) - f(w)| < \frac{\varepsilon}{2}$ for all $w = (w_1, ... ,w_N)$, $|w_j| \leq r$ implies $\chi (Q(w) , f(w)) < \frac{\varepsilon}{2}$, since $\chi (A , B) \leq |A - B|$ for all $A,B \in \mathbb{C}$. The triangle inequality yields the result in the case where $f$ satisfies a') and b') and $f(\overline{D}^I) \subseteq \mathbb{C}$.

If $f(\zeta) = \infty$ for some $\zeta = (\zeta_1, ... ,\zeta_N)$ with $|\zeta_i| < 1$ for all $i$, then an application of Hurwitz's theorem implies that the set $f^{-1}(\{ \infty \}) \cap D^N$is open where $D = \{ z \in \mathbb{C} : |z| < 1\}$. Since it is also closed in the relative topology and $D^N$ is connected, it follows that if $f(\zeta) = \infty$ for some interior point, then $f \equiv \infty$.

In the general case, where $I$ can be infinite, if $f$ satisfies a') and d') and $f(\zeta) = \infty$ for some $\zeta = (\zeta)_{i \in I}$, $|\zeta_i| < 1$ for all $i \in I$, then, by the previous argument $f_{|B} \equiv \infty$, where 

$$B =
\bigcup_{\substack{
   F \subset J \\
   \text{finite set}
  }}
\overline{D}^{F} \times \prod_{j \in J \setminus F} \{ \zeta_j \}
.$$

Since $B$ is dense in $\overline{D}^I$ and $f$ is continuous, it follows that $f \equiv \infty$ and $f$ can be approximated by the constant polynomials $p_n = n$. In the case where $f(\overline D^I)$ is contained in $\mathbb{C}$, then the result follows from the one in the finite case.

\begin{defn}
$\tilde A(\overline D^I)$ contains exactly all functions $f: \overline D^I \to \mathbb{C} \cup \{ \infty \}$ satisfying a') and b') of Proposition $4.1$.
\end{defn}

As we saw $\tilde A( \overline{D}^I)$ coincides with the set of uniform limits of polynomials on $\overline{D}^I$ with respect to the chordal metric $\chi$. Furthermore, if $f \in \tilde A( \overline{D}^I)$ satisfies $f(\zeta) = \infty$ at some $\zeta = (\zeta_i)_{i \in I}$ with $|\zeta_i| < 1$ for all $i \in I$, it follows easily that $f \equiv \infty$. It is also obvious that $A( \overline{D}^I) \subseteq \tilde A(\overline{D}^I)$.

\begin{defn}
A compact set $K \subseteq T^I$ is called a set of uniqueness for the class $\tilde A(\overline{D}^I)$, if any functions $f,g \in \tilde A(\overline{D}^I)$ which coincide on $K$ they automatically coincide on $\overline D^I$.
\end{defn}

Since $A( \overline{D}^I) \subseteq \tilde A(\overline{D}^I)$, it is obvious that any set of uniqueness for $\tilde A(\overline{D}^I)$ is also a set of uniqueness for $A( \overline{D}^I)$. We do not know if the converse holds or not. However, we obtain similar results for the sets of uniqueness for $\tilde A(D)$ as for $A(D)$. The proofs are similar, mainly because by the theorem of Privalov, the following holds: if $f,g \in \tilde A(\overline{D})$ coincide on a subset of $T$ with positive length, then $f \equiv g$. Therefore, we state these results without proof.

\begin{thm}
Let $K \subseteq T^I$ be a compact set with $\lambda(K) > 0$. Then $K$ is a set of uniqueness for $\tilde A(\overline{D}^I)$.
\end{thm}

The set $K$ in the example 3.7 is also good for $\tilde A(\overline{D}^2)$ and $\tilde A(\overline{D}^I)$. Finally cartesian products of sets of uniqueness for $\tilde A(\overline{D}^{I_j})$ are sets of uniqueness for $\tilde A(\overline{D}^I)$ for any partition $(I_j)_{j \in J}$ of $I$; the proof is similar with that of Proposition 3.8.

\section{The case of the Ball of $\mathbb{C}^n$}

In this section we extend the results of Section 3 and 4 replacing $\overline{D}^I$ by the Euclidean ball $\overline{B}_n = \{z = (z_1, ... , z_n) : |z_1^2|+ ... +|z_n^2| \leq 1 \}$ of $\mathbb{C}^n$. Some proofs are similar since the Taylor expansion in several variables converges absolutely and uniformly to the function not only on any closed polydisc contained in the open domain of holomorphy $\Omega$ of the function, but also at any closed concentric ball contained in $\Omega$ (\cite{Rud3.}). Therefore, the set of uniform limits on $\overline{B}_n$ of the polynomials is the algebra $A(\overline{B}_n)$ if the usual Euclidean metric on $\mathbb{C}$ is used and the class $\tilde A(\overline{B}_n)$ if the chordal metric $\chi$ on $\mathbb{C} \cup \{ \infty \}$ is used, defined below.

\begin{defn}
$A(\overline{B}_n)$ contains exactly all functions $f: \overline{B}_n \to \mathbb{C}$ continuous on $\overline{B}_n$ and holomorphic in $B_n = \{z = (z_1, ... , z_n) : |z_1^2|+ ... +|z_n^2| < 1 \}$.
\end{defn}

\begin{defn}
$\tilde A(\overline{B}_n)$ contains exactly all functions $f: \overline{B}_n \to \mathbb{C} \cup \{ \infty \}$ continuous on $\overline{B}_n$, such that for every $z = (z_1, ... ,z_n) \in \overline{B}_n$ and every $k$, $1 \leq k \leq n$ the function
\\ 
$\{w \in \mathbb{C} : |w| \leq \sqrt {1 - (|z_1|^2 + ... +|z_{k-1}|^2+|z_{k+1}|^2+ ... +|z_n|^2)} \} \ni w \to f(z_1, ... ,z_{k-1},w,z_{k+1}, ... , z_n) \in \mathbb{C} \cup \{ \infty \}$
belongs to 
\\
$\tilde A(\{w \in \mathbb{C} : |w| \leq \sqrt {1 - (|z_1|^2 + ... +|z_{k-1}|^2+|z_{k+1}|^2+ ... +|z_n|^2)} \})$.
\end{defn}

We notice that $\tilde A$ of a closed disc with zero radius is meant to be all constants in $\mathbb{C} \cup \{ \infty \}$.

\begin{rem}
If $f \in \tilde A(\overline{B}_n)$  satisfies $f(\zeta) = \infty$ for some $\zeta \in B_n$, then $f = \infty$; otherwise $f(B_n) \subseteq \mathbb{C}$ and $f$ is holomorphic in $B_n$.
\end{rem}

\begin{defn}
Let $K$ be a compact subset of $\partial \overline B_n =$
\\
$= \{ (\zeta_1, ... ,\zeta_n) \in \mathbb{C}^n : |\zeta_1|^2+ ... +|\zeta_n|^2 = 1 \}$. Then $K$ is called a set of uniqueness for $A(\overline B_n)$ (respectively for $\tilde A(\overline B_n)$) if and only if for any functions $f,g \in A(\overline B_n)$(respectively in $\tilde A(\overline B_n)$) which coincide on $K$ then they coincide on $\overline B_n$.

Since $A(\overline B_n) \subseteq \tilde A(\overline B_n)$, it follows that any set of uniqueness for $\tilde A(\overline B_n)$ is also a set of uniqueness for $A(\overline B_n)$. We do not know if the converse holds.
\end{defn}

\begin{rem}
The set $\{ z=(z_1, ... ,z_n) \in \partial B_n : |z_j| > 0 \; \text{for all} \; j = 1, \dots ,n\}$ can be written as 
\\
$\cup_{(r_1, ... ,r_{n-1}) \in G} X_{r_1, ... ,r_{n-1}}$, where 
\\
$X_{r_1, ... ,r_{n-1}} = \{ (r_1e^{i \theta_1}, ... ,r_{n-1}e^{i \theta_{n-1}},r_ne^{i \theta_n}) \in \partial B_n : \theta_j \in \mathbb{R},\;  j = 1, ... ,n\}$,
\\
$G = \{ (r_1, ... ,r_{n-1}) : r_j > 0 \; \text{for all} \; j = 1, ... ,n-1 \; , \; r_1^2+ ... +r_{n-1}^2 < 1 \}$ and
$r_n = \sqrt{1 - (r_1^2+ ... +r_{n-1}^2)}$.
\\
The set $X_{r_1, ... ,r_{n-1}} \cong T^n$ is the  boundary of the polydisc  $\overline D(0,r_1) \times ... \times \overline D(0,r_n)$ which is a subset of $\overline B_n$.
\\
If a compact subset $K \subseteq \partial \overline B_n$ meets $X_{r_1, ... ,r_{n-1}}$ for some $(r_1, ... ,r_{n-1}) \in G$ and the intersection is a set of uniqueness for  $A( \overline D(0,r_1) \times ... \times \overline D(0,r_n))$ (or ($\tilde A( \overline D(0,r_1) \times ... \times \overline D(0,r_n))$) respectively), then the following holds.
\\
If two functions $f,g \in A(\overline B_n)$ (or $f,g \in \tilde A(\overline B_n)$ respectively) coincide on $K$, then according to the result of Sections 3 and 4 they coincide on $\overline D(0,r_1) \times ... \times \overline D(0,r_n)$. It follows easily that $f \equiv g$ on $\overline B_n$. Therefore $K$ is a set of uniqueness for $A(\overline B_n)$ (for $\tilde A(\overline B_n)$ respectively).
\end{rem}

\begin{prop}
If $K \subseteq \partial \overline B_n$ is a compact set with positive measure on the hyper-surface $\partial \overline B_n =\subseteq \mathbb{R}^{2n} = \mathbb{C}^n$ (of real dimension $2n-1$), then $K$ is a set of uniqueness for $A(\overline B_n)$ and $\tilde A(\overline B_n)$.
\end{prop}

\begin{proof}
The natural measure on $\partial \overline B_n$ has the form
\\
$q(r_1, ... ,r_{n-1},e^{i \theta_1}, ... ,e^{i \theta_n}) \mathrm{d} \theta_1 \cdot ... \cdot \mathrm{d} \theta_n \cdot \mathrm{d} r_1 \cdot ... \cdot \mathrm{d} r_{n-1}$ where the function $q$ is a positive $C^{\infty}$ function on $G \times T^n$ and $G = \{ (r_1, \dots ,r_{n-1}) : r_j > 0 \; \text{for} \; j = 1, \dots ,n-1 \; \text{and} \; r_1^2+ ... +r_{n-1}^2 < 1 \}$.
\\
By Fubini's theorem applied to the characteristic function $\chi_K$, it follows that there exists a $(r_1, ... ,r_{n-1}) \in G$, such that $K \cap X_{r_1, ... ,r_{n-1}}$ has positive measure (of dimension $n$).
\\
Thus, by the results of Section 3 and 4 this intersection is a set of uniqueness for the polydiscs $\overline D(0,r_1) \times ... \times \overline D(0,r_n)$. By the previous discussion, it follows that $K$ is a set of uniqueness for $A(\overline B_n)$ (or for  $\tilde A(\overline B_n)$, respectively). This completes the proof.
\end{proof}

Obviously there are plenty of compact sets $K \subseteq \partial \overline B_n$ which are sets of uniqueness for $\tilde A(\overline B_n)$ and have zero measure; for instance, 
\\
let $K = X_(r_1, ... ,r_{n-1})$ for some $(r_1, ... ,r_{n-1}) \in G$.

\section*{Acknowledgements}
We would like to thank G. Alexopoulos, C. Carayiannis, N. Daras, K. Ecker, D. Gatzouras, P.M Gauthier, J. Globevnik, G. Knese, G. Koumoullis, 
\\
M. Manolaki and M. Papatriantafillou for helpful communications.

~\\
~\\
University of Athens
\\
Department of Mathematics
\\
157 84 Panepistimioupolis
\\
Athens
\\
Greece
~\\
~\\
\textbf{email addresses:}
\\
K. Makridis (kmak167@gmail.com)
\\
V. Nestoridis (vnestor@math.uoa.gr)
\end{document}